\documentclass[12pt,draft]{amsart}
\usepackage{amsmath,amssymb,amsthm,bm}
\usepackage{graphicx,enumerate}
\usepackage{layout}

\theoremstyle{plain}
\newtheorem{theorem}{Theorem}[section]
\newtheorem{lemma}[theorem]{Lemma}

\newtheorem{corollary}[theorem]{Corollary}
\newtheorem*{theorem-nn}{Theorem}
\newtheorem*{corollary-nn}{Corollary}

\theoremstyle{definition}

\newtheorem{remark}[theorem]{Remark}

\newcommand{\bZ}{\mathbb{Z}}

\newcommand{\bQ}{\mathbb{Q}}
\newcommand{\opi}{\overline{\pi}}

\title[On the simplest sextic fields and related Thue equations]
{On the simplest sextic fields and related Thue equations}
\author{Akinari Hoshi}
\thanks{This work was partially supported by Rikkyo University Special Fund for Research 
and by the Grant-in-Aid for Young Scientists (B) No. 22740028, 
The Ministry of Education, Culture, Sports, Science and Technology, Japan.}
\begin{document}
\begin{abstract}
We consider the parametric family of sextic Thue equations
\[
x^6-2mx^5y-5(m+3)x^4y^2-20x^3y^3+5mx^2y^4+2(m+3)xy^5+y^6=\lambda
\]
where $m\in\mathbb{Z}$ is an integer and $\lambda$ is a divisor of $27(m^2+3m+9)$. 
We show that the only solutions to the equations are the trivial ones with 
$xy(x+y)(x-y)(x+2y)(2x+y)=0$. 
\end{abstract}
\maketitle

%%%%%%%%%%%%%%%%%%%%%%%%%%%%%%%%%%%%%%%%%%%%%%%%%%%%%%%%%%%%%%%%%%%%%%%%%%%%%%%%%%%%%%%%%%%%
%
\section{Introduction}
%
%%%%%%%%%%%%%%%%%%%%%%%%%%%%%%%%%%%%%%%%%%%%%%%%%%%%%%%%%%%%%%%%%%%%%%%%%%%%%%%%%%%%%%%%%%%%

We consider the following ``simple'' family of sextic Thue equations
\begin{align}
F_m(X,Y):=&\ X^6-2mX^5Y-5(m+3)X^4Y^2\label{eqFm}\\
&-20X^3Y^3+5mX^2Y^4+2(m+3)XY^5+Y^6=\lambda\nonumber
\end{align}
for $m$, $\lambda\in\bZ$ with $\lambda\neq 0$. 
We may assume that $m\geq -1$ because if $F_m(x,y)=\lambda$ then $F_{-m-3}(y,x)=\lambda$. 
If $(x,y)\in\bZ^2$ is a solution to (\ref{eqFm}) then 
\[
(x+y,-x),\ (y,-x-y),\ (-x,-y),\ (-x-y,x),\ (-y,x+y)
\]
are also solutions to (\ref{eqFm}) because $F_m(X,Y)$ is invariant under 
the action of the cyclic group $C_6=\langle\sigma\rangle$ of order $6$ where 
$\sigma : X\mapsto X+Y$, $Y\mapsto -X$. 
For sextic integer $\lambda=e^6$ or $\lambda=-27e^6$, the equation $F_m(X,Y)=\lambda$ has the 
following six solutions respectively 
\begin{align*}
F_m(0,\pm e)&=F_m(\pm e,0)=F_m(\pm e,\mp e)=e^6,\\
F_m(\pm e,\pm e)&=F_m(\pm 2e,\mp e)=F_m(\pm e,\mp 2e)=-27e^6.
\end{align*}
We call such solutions $(x,y)\in\bZ^2$ to $F_m(x,y)=\lambda$ with 
$xy(x+y)(x-y)(x+2y)(2x+y)=0$ the {\it trivial} solutions. 
We remark that $F_m(2x+y,-x+y)=-27F_m(x,y)$. 

For $m\geq 89$, Lettl-Peth\"o-Voutier \cite{LPV99} showed that the only 
primitive solutions $(x,y)\in\bZ^2$, i.e. $\mathrm{gcd}(x,y)=1$, to the 
Thue inequality $|F_m(x,y)|\leq 120m+323$ are 
\begin{align}
F_m(0,\pm 1)&=F_m(\pm 1,0)\hspace*{3.2mm}=F_m(\pm 1,\mp 1)=1,\label{eqtri}\\
F_m(\pm 1,\pm 1)&=F_m(\pm 2,\mp 1)=F_m(\pm 1,\mp 2)=-27\nonumber
\end{align}
(i.e. trivial solutions) and
\begin{align*}
F_m(\pm 1,\pm 2)&=F_m(\pm 3,\mp 1)=F_m(\pm 2,\mp 3)=120m+37,\\
F_m(\pm 2,\pm 1)&=F_m(\pm 3,\mp 2)=F_m(\pm 1,\mp 3)=-120m-323. 
\end{align*}
In \cite{LPV98}, moreover, they obtained that for any $m\in\bZ$ the equation 
$F_m(X,Y)=\lambda$ for $\lambda\in\{\pm 1,\pm 27\}$ has only twelve trivial solutions 
as in (\ref{eqtri}). 
A special case of $F_m(X,Y)=1$ is also studied by Togb\'{e} \cite{Tog02}. 
Wakabayashi \cite{Wak07b} investigated Thue inequalities 
$|F_{l,m}(x,y)|\leq k$ with two parameters $l,m$ and $F_{1,m}=F_m$. 
The following is the main result of this paper (cf. cubic case \cite{H1} 
and quartic case \cite{H2}): 
\begin{theorem}\label{thmain}
For $m\in\bZ$ and a divisor $\lambda$ of $27(m^2+3m+9)$, the only solutions to 
the equation $F_m(x,y)=\lambda$ are the trivial ones with $xy(x+y)(x-y)(x+2y)(2x+y)=0$. 
\end{theorem}

We take the simplest sextic polynomial 
\[
f^{C_6}_m(X):=X^6-2mX^5-5(m+3)X^4-20X^3+5mX^2+2(m+3)X+1\in \bQ[X]
\]
with discriminant $6^6(m^2+3m+9)^5$. 
Note that $f^{C_6}_m(X)=F_m(X,1)$. 

For $m\in\bZ\setminus\{-8,-3,0,5\}$, the polynomial $f^{C_6}_m(X)$ is irreducible 
over $\bQ$ with cyclic Galois group $\mathrm{Gal}_\bQ f^{C_6}_m(X)$ $\cong$ $C_6$ 
of order $6$ (see \cite[Proposition 3.3]{Gra86}). 
The splitting fields 
\[
L^{(6)}_m:=\mathrm{Spl}_\bQ f^{C_6}_m(X),\quad (m\in\bZ\setminus\{-8,-3,0,5\})
\]
of $f^{C_6}_m(X)$ over $\bQ$ are totally real cyclic sextic fields and 
called the simplest sextic fields 
(cf. e.g. \cite{Gra86}, \cite{Gra87}, \cite{LPV98}, \cite{LPV99}, 
\cite[Section 8.3]{Gaa02}, \cite{Tog02}, \cite{HH05}, \cite{Lou07}). 

We get Theorem \ref{thmain} as a consequence of the following two theorems: 

\begin{theorem-nn}[Theorem \ref{thisoC6}]
For $m,n\in\bZ$, $L^{(6)}_m=L^{(6)}_n$ if and only if $m=n$ or $m=-n-3$. 
\end{theorem-nn}

\begin{theorem-nn}[Theorem \ref{th2}]
There exists an integer $n\in\bZ\setminus\{m,-m-3\}$ such that 
$L^{(6)}_n=L^{(6)}_m$ if and only if there exists non-trivial solution 
$(x,y)\in\bZ^2$, i.e. $xy(x+y)(x-y)(x+2y)(2x+y)\neq 0$, to $F_m(x,y)=\lambda$ 
where $\lambda$ is a divisor of $27(m^2+3m+9)$. 
\end{theorem-nn}

In Section \ref{sesimsextic}, 
we review some facts on the simplest sextic and cubic fields. 
In Section \ref{seInt}, we recall known results of the resolvent polynomials 
which are fundamental tools in the computational aspects of Galois theory. 
We intend to explain how to obtain an answer to the field intersection 
problem of cyclic sextic polynomials $f_s(X)$ over a field $K$ of 
char $K\neq 2,3$, i.e. for $a,b\in K$ how to determine the intersection 
of $\mathrm{Spl}_K f_a(X)$ and $\mathrm{Spl}_K f_b(X)$. 
In Section \ref{seIso}, we give an explicit answer to the field isomorphism 
problem of $f_s^{C_6}(X)$ as the special case of the field intersection problem. 
In particular, for $K=\bQ$, we get Theorem \ref{thisoC6} by using Okazaki's theorem 
(Theorem \ref{thOka}). 
In Section \ref{secores}, we will show a correspondence between isomorphism 
classes of the simplest sextic fields $L^{(6)}_m$ and non-trivial solutions to 
the sextic Thue equations $F_m(x,y)=\lambda$ where $\lambda$ is a divisor of $27(m^2+3m+9)$ 
(see Theorem \ref{th2}). 

%%%%%%%%%%%%%%%%%%%%%%%%%%%%%%%%%%%%%%%%%%%%%%%%%%%%%%%%%%%%%%%%%%%%%%%%%%%%%%%%%%%%%%%%%%%%
%
\section{The simplest sextic and cubic fields}\label{sesimsextic}
%
%%%%%%%%%%%%%%%%%%%%%%%%%%%%%%%%%%%%%%%%%%%%%%%%%%%%%%%%%%%%%%%%%%%%%%%%%%%%%%%%%%%%%%%%%%%%

We recall known facts of the simplest sextic and cubic fields 
(see \cite[Section 3]{Gra86}, \cite{Gra87}, \cite{LPV98}). 

Let $K$ be a field of char $K\neq 2, 3$ and $K(s)$ the rational function 
field over $K$ with variable $s$. 
We take the simplest sextic polynomial 
\begin{align*}
f^{C_6}_s(X):=X^6-2sX^5-5(s+3)X^4-20X^3+5sX^2+2(s+3)X+1
\end{align*}
with discriminant $6^6(s^2+3s+9)^5$. 
The Galois group of $f^{C_6}_s(X)$ over $K(s)$ is isomorphic to the cyclic 
group $C_6$ of order $6$. 
Gras \cite{Gra86} considered the polynomial 
\[
g_t(X)=X^6-\frac{1}{2}(t-6)X^5-\frac{5}{4}(t+6)X^4-20X^3
+\frac{5}{4}(t-6)X^2+\frac{1}{2}(t+6)X+1.
\]
The two polynomials above are related by $g_{4s+6}(X)=f^{C_6}_s(X)$. 

Let $K(z)$ be the rational function field over $K$ with variable $z$ and 
$\sigma$ a $K$-automorphism of $K(z)$ of order $6$ which is defined by 
\begin{align}
\sigma : z\ \mapsto\ \frac{z-1}{z+2}\ \mapsto\ -\frac{1}{z+1}\ 
\mapsto\ -\frac{z+2}{2z+1}\ \mapsto\ -\frac{z+1}{z}\ \mapsto\ 
-\frac{2z+1}{z-1}\ \mapsto\ z.\label{actz} 
\end{align}
Then we get the Galois extension $K(z)/K(z)^{\langle\sigma\rangle}$ 
with cyclic Galois group $C_6$ of order $6$ and 
\begin{align*}
f^{C_6}_s(X)&=\prod_{x\in\mathrm{Orb}_{\langle\sigma\rangle}(z)}\Bigl(X-x\Bigr)
\end{align*}
where
\begin{align*}
s=\frac{z^6-15z^4-20z^3+6z+1}{z(2z^4+5z^3-5z-2)}
=\frac{(z^3+3z^2-1)(z^3-3z^2-6z+1)}{z(z+1)(z-1)(z+2)(2z+1)}
\end{align*}
as the generating polynomial of the sextic cyclic field $K(z)$ 
over $K(z)^{\langle\sigma\rangle}=K(s)$. 

The quadratic field $K(z)^{\langle\sigma^2\rangle}$ over $K(s)$ is given by 
$K(s)(z_2)$ where
\[
z_2=z+\sigma^2(z)+\sigma^4(z)=\frac{z^3-3z-1}{z(z+1)}. 
\]
It also follows from 
\begin{align}
z_2-s=\frac{(z^2+z+1)^3}{z(z+1)(z-1)(z+2)(2z+1)}=\sqrt{s^2+3s+9}\label{eqz2s}
\end{align}
that $K(z)^{\langle\sigma^2\rangle}=K(s)(\sqrt{s^2+3s+9})$. 

The cyclic cubic field $K(z)^{\langle\sigma^3\rangle}$ over $K(s)$ is given by 
$K(s)(z_3)$ where 
\[
z_3=\frac{1}{z\,\sigma^3(z)}=-\frac{z(z+2)}{(z+1)(z-1)}.
\]
The action of $\sigma$ on $K(s)(z_3)$ is given by 
\[
\sigma : s\ \mapsto\ s,\ z_3\ \mapsto\ -\frac{1}{z_3+1}\ \mapsto\ 
-\frac{z_3+1}{z_3}\ \mapsto\ z_3. 
\]
Hence the minimal polynomial of $z_3$ over $K(s)$ is given by 
\begin{align*}
f_s^{C_3}(X):=\prod_{x\in\mathrm{Orb}_{\langle\sigma\rangle}(z_3)}\Bigl(X-x\Bigr)
=X^3-sX^2-(s+3)X-1,
\end{align*}
that is the simplest cubic polynomial of Shanks \cite{Sha74}. 
Two polynomials $f^{C_6}_s(X)$ and $f_s^{C_3}(X)$ satisfy the relation 
\begin{align}
f^{C_6}_s(X)=(f_s^{C_3}(X))^2-(s^2+3s+9)X^2(X+1)^2.\label{eqC6C3}
\end{align}

For $K=\bQ$, we consider the specialization map $s\mapsto m\in\bZ$. 
By \cite[Proposition 3.3]{Gra86}, the sextic polynomial $f^{C_6}_m(X)$ is 
irreducible over $\bQ$ for $m\in\bZ\setminus\{-8,-3,0,5\}$, 
and the splitting fields $L^{(6)}_m=\mathrm{Spl}_\bQ f^{C_6}_m(X)$ are totally 
real cyclic number fields of degree $6$ which are called the simplest sextic fields. 
We see $L_m^{(6)}=L_{-m-3}^{(6)}$ because if $z$ is a root of $f^{C_6}_m(X)$ 
then $1/z$ becomes a root of $f^{C_6}_{-m-3}(X)$.

For $m\in\{-8,-3,0,5\}$, an integer $m^2+3m+9$ is square and then 
from (\ref{eqz2s}) and (\ref{eqC6C3}) the sextic polynomial 
$f^{C_6}_m(X)$ splits over $\bQ$. 
Indeed we see 
\begin{align*}
f^{C_6}_{-8}(X)&=f_{-1}^{C_3}(X)f_{-15}^{C_3}(X),& 
f^{C_6}_{-3}(X)&=f_{0}^{C_3}(X)f_{-6}^{C_3}(X),\\
f^{C_6}_{0}(X)&=f_{3}^{C_3}(X)f_{-3}^{C_3}(X),& 
f^{C_6}_{5}(X)&=f_{12}^{C_3}(X)f_{-2}^{C_3}(X). 
\end{align*}

The cubic subfields $L_m^{(3)}:=\mathrm{Spl}_\bQ f_m^{C_3}(X)$ of $L^{(6)}_m$ 
are called the simplest cubic fields. 
Note that $f_m^{C_3}(X)$ is irreducible over $\bQ$ 
and $L_m^{(3)}=L_{-m-3}^{(3)}$ for any $m\in\bZ$. 

Ennola \cite{Enn91} verified that for integers $-1\leq m<n\leq 10^4$, 
$L_m^{(3)}=L_n^{(3)}$ if and only if $(m,n)\in \{(-1,5)$, $(-1,12)$, $(-1,1259)$, 
$(5,12)$, $(5,1259)$, $(12,1259)\}$ $\cup$ $\{(0,3)$, $(0,54)$, $(3,54)\}$ 
$\cup$ $\{(1,66)\}$ $\cup$ $\{(2,2389)\}$. 
Hoshi-Miyake \cite[Example 5.3]{HM09a} checked that Ennola's claim is also 
valid for $-1\leq m<n\leq 10^5$. 

In \cite{Oka02}, Okazaki investigated Thue equations $F(X,Y)=1$ for irreducible 
cubic forms $F$ with positive discriminant $D(F)>0$ and established a very 
strong result on gaps between solutions (cf. also \cite{Wak07a}). 
By using methods in \cite{Oka02}, we may obtain that 
if $L^{(3)}_m=L^{(3)}_n$ with $-1\leq m<n$ then $m\leq 35731$ (cf. also \cite{H1}). 
Moreover Okazaki showed the following theorem: 
\begin{theorem}[Okazaki {\cite{Oka}}]\label{thOka}
Let $L_m^{(3)}=\mathrm{Spl}_\bQ f^{C_3}_m(X)$. 
For $m,n\in\bZ$ with $-1\leq m<n$, if $L^{(3)}_m=L^{(3)}_n$ then 
$m,n\in\{-1$, $0$, $1$, $2$, $3$, $5$, $12$, $54$, $66$, $1259$, $2389\}$. 
In particular, 
\[
L^{(3)}_{-1}=L^{(3)}_5=L^{(3)}_{12}=L^{(3)}_{1259},\quad 
L^{(3)}_0=L^{(3)}_3=L^{(3)}_{54},\quad 
L^{(3)}_1=L^{(3)}_{66},\quad 
L^{(3)}_2=L^{(3)}_{2389}. 
\]
\end{theorem}
The author also gave an another proof of Theorem \ref{thOka} in \cite{H1}. 

%%%%%%%%%%%%%%%%%%%%%%%%%%%%%%%%%%%%%%%%%%%%%%%%%%%%%%%%%%%%%%%%%%%%%%%%%%%%%%%%%%%%%%%%%%%%
%
\section{Field intersection problem of cyclic sextic}\label{seInt}
%
%%%%%%%%%%%%%%%%%%%%%%%%%%%%%%%%%%%%%%%%%%%%%%%%%%%%%%%%%%%%%%%%%%%%%%%%%%%%%%%%%%%%%%%%%%%%

We recall some results of the resolvent polynomials which are fundamental tools 
in the computational aspects of Galois theory (cf. e.g. \cite{Coh93}, \cite{Coh00}, 
\cite{Ade01}). 
Let $\overline{K}$ be a fixed algebraic closure of a field $K$.
Let $f(X):=\prod_{i=1}^m(X-\alpha_i) \in K[X]$ be a separable polynomial of degree 
$m$ with some fixed order of the roots $\alpha_1,\ldots,\alpha_m\in \overline{K}$. 
The Galois group of the polynomial $f(X)$ over $K$ may be determined by resolvent 
polynomials with suitable invariants. 

Let $R:=K[x_1,\ldots,x_m]$ be the polynomial ring over $K$ with $m$ variables 
$x_1,\ldots,x_m$. 
For $\Theta \in R$, we take the specialization map 
$\omega_f : R \rightarrow k(\alpha_1,\ldots,\alpha_m)$, 
$\Theta(x_1,\ldots,x_m)\mapsto \Theta(\alpha_1,\ldots,\alpha_m)$. 
The kernel of $\omega_f$ is the ideal 
$I_f:=\{\Theta\in R \mid \Theta(\alpha_1,\ldots,\alpha_m)=0\}$ in $R$. 

Let $S_m$ be the symmetric group of degree $m$. 
We extend the action of $S_m$ on $m$ letters $\{1,\ldots,m\}$ to that on $R$ by 
$\pi(\Theta(x_1,\ldots,x_m)):=\Theta(x_{\pi(1)},\ldots,x_{\pi(m)})$. 
The Galois group of $f(X)$ over $K$ is defined by 
$\mathrm{Gal}(f/K):={\{\pi\in S_m \mid \pi(I_f)\subseteq I_f\}}$, 
and $\mathrm{Gal}(f/K)$ is isomorphic to the Galois group of the splitting field 
$\mathrm{Spl}_K f(X)$ of $f(X)$ over $K$. 
If we take another ordering of roots $\alpha_{\pi(1)},\ldots,\alpha_{\pi(m)}$ 
of $f(X)$ for some $\pi\in S_m$, the corresponding realization of 
$\mathrm{Gal}(f/K)$ is conjugate in $S_m$. 
Hence, for arbitrary ordering of the roots of $f(X)$, $\mathrm{Gal}(f/K)$ is 
determined up to conjugacy in $S_m$. 

For $H\leq U\leq S_m$, an element $\Theta\in R$ is called a 
$U$-primitive $H$-invariant if 
$H=\mathrm{Stab}_U(\Theta)$ $:=$ $\{\pi\in U\ |\ \pi(\Theta)=\Theta\}$. 
For a $U$-primitive $H$-invariant $\Theta$, the polynomial 
\[
\mathcal{RP}_{\Theta,U}(X)=\prod_{\opi\in U/H}(X-\pi(\Theta))\in R^U[X]
\]
where $\opi$ runs through the left cosets of $H$ in $U$, 
is called the {\it formal} $U$-relative $H$-invariant resolvent by $\Theta$. 
The polynomial 
\[
\mathcal{RP}_{\Theta,U,f}(X):=\omega_f(\mathcal{RP}_{\Theta,U}(X))
\]
is called the $U$-relative $H$-invariant resolvent of $f$ by $\Theta$. 
The following theorem is fundamental in the theory of resolvent polynomials 
(see e.g. \cite[p.95]{Ade01}). 

\begin{theorem}\label{thfun}
Let $G=\mathrm{Gal}(f/K)$, $H\leq U\leq S_m$ be finite groups with $G\leq U$ 
and $\Theta$ a $U$-primitive  $H$-invariant. 
Suppose that $\mathcal{RP}_{\Theta,U,f}(X)=\prod_{i=1}^l h_i^{e_i}(X)$ gives 
the decomposition of $\mathcal{RP}_{\Theta,U,f}(X)$ into a product of powers of 
distinct irreducible polynomials $h_i(X)$, $(i=1,\ldots,l)$, in $K[X]$. 
Then we have a bijection 
\begin{align*}
G\backslash U/H\quad &\longrightarrow \quad \{h_1^{e_1}(X),\ldots,h_l^{e_l}(X)\}\\
G\, \pi\, H\quad &\longmapsto\quad h_\pi(X)
=\prod_{\tau H\subseteq G\,\pi\,H}\bigl(X-\omega_{f}(\tau(\Theta))\bigr)
\end{align*}
where the product runs through the left cosets $\tau H$ of $H$ in $U$ contained 
in $G\, \pi\, H$, that is, through $\tau=\pi_\sigma \pi$ where $\pi_\sigma$ runs 
a system of representative of the left cosets of $G \cap \pi H\pi^{-1};$ each 
$h_\pi(X)$ is irreducible or a power of an irreducible polynomial with 
$\mathrm{deg}(h_\pi(X))$ $=$ $|G\, \pi\, H|/|H|$ $=$ $|G|/|G\cap \pi H\pi^{-1}|$. 
\end{theorem}
\begin{corollary} 
If $G\leq \pi H\pi^{-1}$ for some $\pi\in U$ then $\mathcal{RP}_{\Theta,U,f}(X)$ 
has a linear factor over $K$. 
Conversely, if $\mathcal{RP}_{\Theta,U,f}(X)$ has a non-repeated linear factor 
over $K$ then there exists $\pi\in U$ such that $G\leq \pi H\pi^{-1}$. 
\end{corollary}
\begin{remark}
When $\mathcal{RP}_{\Theta,U,f}(X)$ is not squarefree, there exists 
a suitable Tschirnhausen transformation $\hat{f}$ of $f$ over $K$ such that 
$\mathcal{RP}_{\Theta,U,\hat{f}}(X)$ is squarefree (cf. \cite{Gir83}, 
\cite[Alg. 6.3.4]{Coh93}). 
\end{remark}

Now we apply Theorem \ref{thfun} to the cyclic sextic case. 
Let $f^1(X)$, $f^2(X)\in K[X]$ be separable sextic polynomials over $K$. 
We put 
\[
f(X):=f^1(X)f^2(X),\quad G_i:=\mathrm{Gal}(f^i/K),\quad L_i:=\mathrm{Spl}_K f^i(X),\quad 
(i=1,2).
\]

We assume that $G_1,G_2\leq C_6$ and apply Theorem \ref{thfun} to $m=12$, 
$f(X)=f^1(X)f^2(X)$, $U=\langle\sigma\rangle\times\langle\tau\rangle$, 
$H=\langle\sigma\tau\rangle$ or $\langle\sigma\tau^5\rangle$ where 
$\sigma,\tau\in S_{12}$ act on $R=K[x_1,\ldots,x_{12}]$ by 
\begin{align*}
\sigma \,:\, x_1\mapsto x_2\mapsto \cdots \mapsto x_6 \mapsto x_1,\quad 
\tau \,:\, x_7\mapsto x_8\mapsto \cdots \mapsto x_{12}\mapsto x_7.
\end{align*}
Let $\Theta_1$ (resp. $\Theta_2$) be a $U$-primitive 
$\langle\sigma\tau\rangle$-invariant (resp. $\langle\sigma\tau^5\rangle$-invariant) 
where $U=\langle\sigma\rangle\times\langle\tau\rangle$. 
Then the $U$-relative $\langle\sigma\tau\rangle$-invariant 
(resp. $\langle\sigma\tau^5\rangle$-invariant) resolvent 
polynomial of $f(X)=f^1(X)f^2(X)$ by $\Theta_1$ (resp. $\Theta_2$) is given by 
\begin{align*}
\mathcal{R}_f^i(X):=\mathcal{RP}_{\Theta_i,U,f}(X),\quad (i=1,2).
\end{align*}
This is also called (absolute) {\it multi-resolvent polynomial} 
(cf. \cite{RV99}, \cite{Ren04}). 

For a squarefree polynomial $\mathcal{R}(X)\in K[X]$ of degree $l$, we define 
the {\it decomposition type} $\mathrm{DT}(\mathcal{R})$ of $\mathcal{R}(X)$ 
by the partition of $l$ induced by the degrees of the irreducible factors of 
$\mathcal{R}(X)$ over $K$. 
By Theorem \ref{thfun}, we get the intersection field $L_1$ $\cap$ $L_2$ via the decomposition types 
$\mathrm{DT}(\mathcal{R}_f^1)$ and $\mathrm{DT}(\mathcal{R}_f^2)$. 
\begin{theorem}\label{thDecom}
For $f(X)=f^1(X)f^2(X)\in K[X]$ with $G_1$, $G_2\leq C_6$, we assume that 
$\#G_1\geq\#G_2$ and both $\mathcal{R}_f^1(X)$ and $\mathcal{R}_f^2(X)$ are 
squarefree. 
Then the Galois group $G=\mathrm{Gal}(f/K)$ and the intersection field 
$L_1 \cap L_2$ are given by 
the decomposition types $\mathrm{DT}(\mathcal{R}_f^1)$ and 
$\mathrm{DT}(\mathcal{R}_f^2)$ as on Table $1$. 
\end{theorem}
\begin{center}
{\rm Table} $1$\vspace*{5mm}\\
\begin{tabular}{|c|c|c|l|l|l|}\hline
$G_1$& $G_2$ & $G$ & & ${\rm DT}(\mathcal{R}_f^1)$ 
& ${\rm DT}(\mathcal{R}_f^2)$ \\ \hline 
%%%%%%%%%%%%%%%%%%%%%%%%%%%%%%%%%%%%%%%%%%%%%%%%%%%%%%%%%%%%%%%%%%%%%%%%%
& & $C_6\times C_6$ & $L_1\cap L_2=K$ & $6$ & $6$\\\cline{3-6} 
& & $C_6\times C_3$ & $[L_1\cap L_2:K]=2$ & $3,3$ & $3,3$\\ \cline{3-6} 
& \raisebox{-1.6ex}[0cm][0cm]{$C_6$} & \raisebox{-1.6ex}[0cm][0cm]{$C_6\times C_2$} 
& \raisebox{-1.6ex}[0cm][0cm]{$[L_1\cap L_2:K]=3$} & $6$ & $2,2,2$\\ \cline{5-6}
& & & & $2,2,2$ & $6$\\ \cline{3-6} 
& & \raisebox{-1.6ex}[0cm][0cm]{$C_6$} & \raisebox{-1.6ex}[0cm][0cm]{$L_1=L_2$} 
& $3,3$ & $1,1,1,1,1,1$\\ \cline{5-6}
$C_6$ & & & & $1,1,1,1,1,1$ & $3,3$\\ \cline{2-6} 
& & $C_6\times C_3$ & $L_1\cap L_2=K$ & $6$ & $6$\\ \cline{3-6} 
& $C_3$ & \raisebox{-1.6ex}[0cm][0cm]{$C_6$} 
& \raisebox{-1.6ex}[0cm][0cm]{$L_1\supset L_2$} & $6$ & $2,2,2$\\ \cline{5-6}
& & & & $2,2,2$ & $6$\\ \cline{2-6} 
& \raisebox{-1.6ex}[0cm][0cm]{$C_2$} & $C_6\times C_2$ & $L_1\cap L_2=K$ 
& $6$ & $6$\\ \cline{3-6}
& & $C_6$ & $L_1\supset L_2$ & $3,3$ & $3,3$\\ \cline{2-6} 
& $\{1\}$  & $C_6$ & $L_1\supset L_2=K$ & $6$ & $6$\\ \cline{1-6}
%%%%%%%%%%%%%%%%%%%%%%%%%%%%%%%%%%%%%%%%%%%%%%%%%%%%%%%%%%%%%%%%%%%%%%%%%
& & $C_3\times C_3$ & $L_1\cap L_2=K$ & $3,3$ & $3,3$\\\cline{3-6} 
& $C_3$ & \raisebox{-1.6ex}[0cm][0cm]{$C_3$} 
& \raisebox{-1.6ex}[0cm][0cm]{$L_1=L_2$} & $3,3$ & $1,1,1,1,1,1$\\ \cline{5-6}
$C_3$ & & & & $1,1,1,1,1,1$ & $3,3$\\ \cline{2-6} 
& $C_2$ & $C_6$ & $L_1\cap L_2=K$ & $6$ & $6$\\ \cline{2-6} 
& $\{1\}$  & $C_3$ & $L_1\supset L_2=K$ & $3,3$ & $3,3$\\ \cline{1-6}
%%%%%%%%%%%%%%%%%%%%%%%%%%%%%%%%%%%%%%%%%%%%%%%%%%%%%%%%%%%%%%%%%%%%%%%%%
& \raisebox{-1.6ex}[0cm][0cm]{$C_2$}  & $C_2\times C_2$ & $L_1\cap L_2=K$ 
& $2,2,2$ & $2,2,2$\\ \cline{3-6}
$C_2$ & & $C_2$ & $L_1=L_2$ & $1,1,1,1,1,1$ & $1,1,1,1,1,1$\\ \cline{2-6}
& \{1\} & $C_2$ & $L_1\supset L_2$ & $2,2,2$ & $2,2,2$\\ \cline{1-6}
%%%%%%%%%%%%%%%%%%%%%%%%%%%%%%%%%%%%%%%%%%%%%%%%%%%%%%%%%%%%%%%%%%%%%%%%%
$\{1\}$ & $\{1\}$ & $\{1\}$ & $L_1=L_2=K$ & $1,1,1,1,1,1$ & $1,1,1,1,1,1$\\ \cline{1-6}
%%%%%%%%%%%%%%%%%%%%%%%%%%%%%%%%%%%%%%%%%%%%%%%%%%%%%%%%%%%%%%%%%%%%%%%%%
\end{tabular}
\vspace*{5mm}
\end{center}

We checked the decomposition types $\mathrm{DT}(\mathcal{R}_f^i)$, $(i=1,2)$, 
on Table $1$ using the computer algebra system GAP \cite{GAP} via the command 
{\tt DoubleCosetRepsAndSizes}. 

%%%%%%%%%%%%%%%%%%%%%%%%%%%%%%%%%%%%%%%%%%%%%%%%%%%%%%%%%%%%%%%%%%%%%%%%%%%%%%%%%%%%%%%%%%%%
%
\section{An explicit answer to the isomorphism problem}\label{seIso}
%
%%%%%%%%%%%%%%%%%%%%%%%%%%%%%%%%%%%%%%%%%%%%%%%%%%%%%%%%%%%%%%%%%%%%%%%%%%%%%%%%%%%%%%%%%%%%

By using Theorem \ref{thDecom}, 
we give an answer to the field intersection problem of 
\[
f^{C_6}_s(X)=X^6-2sX^5-5(s+3)X^4-20X^3+5sX^2+2(s+3)X+1,
\]
i.e. for $a,b\in K$ how to determine the intersection of $\mathrm{Spl}_K f_a(X)$ 
and $\mathrm{Spl}_K f_b(X)$, via multi-resolvent polynomials. 
An explicit answer to the field isomorphism problem of $f^{C_6}_s(X)$ will be 
given as the special case of the field intersection problem (see Theorem \ref{thC6}). 

For $n\geq 3$, Rikuna \cite{Rik02} constructed one-parameter families of cyclic 
polynomials $f^{R(n)}_s(X)$ of degree $n$ over $K$ with char $K {\not|}$ $n$ and 
$K\ni \zeta_n+\zeta_n^{-1}$ where $\zeta_n$ is a primitive $n$-th root of unity. 
The simplest sextic polynomial $f_s^{C_6}(X)$ may be obtained as the sextic case 
$f^{R(6)}_s(X)$ (see also \cite{Miy99}, \cite{HM99}). 

Komatsu \cite{Kom04} established descent Kummer theory via $f^{R(n)}_s(X)$, and 
gave a necessary and sufficient condition to 
$\mathrm{Spl}_K f^{R(n)}_a(X)\subset \mathrm{Spl}_K f^{R(n)}_b(X)$ for $a,b\in K$ 
(see also \cite{Oga03}, \cite{Kid05}). 
It is interesting to compare the results of \cite{Oga03}, \cite{Kom04}, 
\cite{Kid05} with results given in this section. 
We note that a method via multi-resolvent polynomials is valid also for 
non-abelian groups (see \cite{HM07}, \cite{HM09b}, \cite{HM09c}, \cite{HM10c}). 

Let $K$ be a field of char $K\neq 2,3$, 
$K(z)$ the rational function field over $K$ with variable $z$ and 
$\sigma$ a $K$-automorphism of $K(z)$ of order $6$ which is given by (\ref{actz}). 
We also take another rational function field $K(w)$ over $K$, a $K$-automorphism 
of $K(w)$ of order $6$ 
\begin{align*}
\tau : w\ \mapsto\ \frac{w-1}{w+2}\ \mapsto\ -\frac{1}{w+1}\ \mapsto\ 
-\frac{w+2}{2w+1}\ \mapsto\ -\frac{w+1}{w}\ \mapsto\ -\frac{2w+1}{w-1}\ \mapsto\ w
\end{align*}
and $f^{C_6}_t(X)=X^6-2tX^5-5(t+3)X^4-20X^3+5tX^2+2(t+3)X+1$ where 
\begin{align*}
t=\frac{w^6-15w^4-20w^3+6w+1}{w(2w^4+5w^3-5w-2)}
=\frac{(w^3+3w^2-1)(w^3-3w^2-6w+1)}{w(w+1)(w-1)(w+2)(2w+1)}
\end{align*}
by the same manner of $K(z)$, $\sigma$ and $f_s^{C_6}(X)$. 

Then the field $K(z,w)$ is $(C_6\times C_6)$-extension of $K(z,w)^U=K(s,t)$ 
where $U=\langle\sigma\rangle\times\langle\tau\rangle$. 
Now we should find suitable $U$-primitive $\langle\sigma\tau\rangle$-invariant 
$\Theta_1$ (resp. $\langle\sigma\tau^5\rangle$-invariant $\Theta_2$). 
By \cite[Theorem 1.4]{AHK98}, there exists $\langle\sigma\tau\rangle$-invariant 
$\Theta_1$ such that $K(z,w)=K(z,\Theta_1)$. 
Moreover we may obtain the following $\Theta_1$ and $\Theta_2$: 
\begin{lemma}\label{lemTheta}
Let $U=\langle\sigma\rangle\times\langle\tau\rangle$, 
\[
\Theta_1=-\frac{zw+z+1}{z-w}\quad\text{and}\quad 
\Theta_2=\frac{zw-1}{z+w+1}.
\]
Then the following assertions hold$\,:$\\
{\rm (i)} the element $\Theta_1$ is a $U$-primitive 
$\langle\sigma\tau\rangle$-invariant$\,;$\\
{\rm (ii)} the element $\Theta_2$ is a $U$-primitive 
$\langle\sigma\tau^5\rangle$-invariant$\,;$\\
{\rm (iii)} the $U$-orbit of $\Theta_i$ is given by 
the same as $\langle\sigma\rangle$-orbit of $z\,;$ 
\begin{align*}
\mathrm{Orb}_U(\Theta_i)=
\Big\{\Theta_i, \frac{\Theta_i-1}{\Theta_i+2},-\frac{1}{\Theta_i+1},
-\frac{\Theta_i+2}{2\Theta_i+1},-\frac{\Theta_i+1}{\Theta_i},
-\frac{2\Theta_i+1}{\Theta_i-1}\Big\},\quad (i=1,2). 
\end{align*}
\end{lemma}
\begin{proof}
We can check the assertions by direct computations. 
\end{proof}

Put $f_{a,b}(X):=f^{C_6}_a(X)f^{C_6}_b(X)$. 
The multi-resolvent polynomials 
\begin{align*}
\mathcal{R}_{f_{a,b}}^i(X):=
\mathcal{RP}_{\Theta_i,\langle\sigma\rangle\times\langle\tau\rangle,f_{a,b}}(X),
\quad (i=1,2)
\end{align*}
with respect to $\Theta_1$ and $\Theta_2$ as in Lemma \ref{lemTheta} are given by 
\begin{align}
\mathcal{R}_{f_{a,b}}^i(X)=f_{A_i}^{C_6}(X),\quad (i=1,2)\label{polyR}
\end{align}
where
\begin{align*}
A_1=-\frac{ab+3a+9}{a-b},\quad A_2=\frac{ab-9}{a+b+3}.
\end{align*}
Note that 
\begin{align*}
\mathrm{disc}(\mathcal{R}_{f_{a,b}}^1(X))&=
\frac{6^6(a^2+3a+9)^5(b^2+3b+9)^5}{(a-b)^{10}},\\
\mathrm{disc}(\mathcal{R}_{f_{a,b}}^2(X))&=
\frac{6^6(a^2+3a+9)^5(b^2+3b+9)^5}{(a+b+3)^{10}}.
\end{align*}

By Theorem \ref{thDecom}, we get the intersection field 
$\mathrm{Spl}_K f^{C_6}_a(X)$ $\cap$ $\mathrm{Spl}_K f^{C_6}_b(X)$ via Table $1$. 

\begin{theorem}\label{thDab}
Let $K$ be a field of char $K\neq 2,3$ and $\mathcal{R}_{f_{a,b}}^i(X)=f_{A_i}^{C_6}(X)$, 
$(i=1,2)$, as in $(\ref{polyR})$. 
For $a,b\in K$ with $(a-b)(a+b+3)\neq 0$ and $(a^2+3a+9)(b^2+3b+9)\neq 0$, 
we assume that $\#\mathrm{Gal}_K f^{C_6}_a(X)\geq\#\mathrm{Gal}_K f^{C_6}_b(X)$. 
Then the intersection field 
$\mathrm{Spl}_K f^{C_6}_a(X)\cap\mathrm{Spl}_K f^{C_6}_b(X)$ is given by the 
decomposition types $\mathrm{DT}(\mathcal{R}_{f_{a,b}}^1)$ and 
$\mathrm{DT}(\mathcal{R}_{f_{a,b}}^2)$ as on Table $1$.
\end{theorem}

As the special case of Theorem \ref{thDab}, we obtain an explicit answer to 
the field isomorphism problem of $f_s^{C_6}(X)$. 
\begin{theorem}\label{thC6}
Let $K$ be a field of char $K\neq 2,3$. 
For $a,b\in K$ with $(a-b)(a+b+3)\neq 0$ and $(a^2+3a+9)(b^2+3b+9)\neq 0$, 
the following three conditions are equivalent$\,:$\\
{\rm (i)} the splitting fields of $f^{C_6}_a(X)$ and of $f^{C_6}_b(X)$ 
over $K$ coincide$\,;$\\
{\rm (ii)} the polynomial $f^{C_6}_{A_i}(X)$ splits completely into 
$6$ linear factors over $K$ for $i=1$ or $i=2$ where
\[
A_1=-\frac{ab+3a+9}{a-b}\ \ \mathit{and}\ \ A_2=\frac{ab-9}{a+b+3}\,;
\]
{\rm (iii)} there exists $z\in K$ such that 
\[
B\,=\,a+\frac{(a^2+3a+9)z(z+1)(z-1)(z+2)(2z+1)}{f_a(z)}
\]
where $B=b$ or $B=-b-3$. 

Moreover if $\mathrm{Gal}_K f^{C_6}_a(X)\cong C_6$ or $C_3$ $($resp. 
$\mathrm{Gal}_K f^{C_6}_a(X)\cong C_2$ or $\{1\})$ then {\rm (ii)} occurs for 
only one of $A_1$ and $A_2$ $($resp. for both of $A_1$ and $A_2)$ and {\rm (iii)} 
occurs for only one of $b$ and $-b-3$ $($resp. for both of $b$ and $-b-3)$. 
\end{theorem}
\begin{remark}
The condition (iii) is just a restatement of (ii). 
Indeed, for $i=1,2$, rational roots $z\in K$ of $f_{A_i}^{C_6}(X)$ satisfy 
the condition (iii) for $B=b$ and $B=-b-3$ respectively. 
The equivalence of the conditions (i) and (iii) is valid also for $a=b$ and $b=-a-3$. 
\end{remark}

Theorem \ref{thC6} is a generalization of the results of the simplest cubic 
(resp. quartic) case in \cite{Mor94}, \cite{Cha96}, \cite{HM09a} (resp. \cite{H2}). 
This is an analogue of Kummer theory; for a field $K$ which contains 
a primitive $6$th root $\zeta_6$ of unity and $a,b\in K$, 
$\mathrm{Spl}_K(X^6-a)=\mathrm{Spl}_K(X^6-b)$ if and only if 
$X^6-ab$ or $X^6-ab^5$ splits completely over $K$. 
It is remarkable that Theorem \ref{thC6} does not need the assumption that 
$K$ contains $\zeta_6$. 

By Theorem \ref{thC6}, for a fixed $a\in K$ with $a^2+3a+9\neq 0$, we have 
$\mathrm{Spl}_K f^{C_6}_b(X)=\mathrm{Spl}_K f^{C_6}_a(X)$ where $b$ is given 
as in Theorem \ref{thC6} (iii) for arbitrary $z\in K$ with $f_a(z)\neq 0$ 
and $b^2+3b+9\neq 0$. 
\begin{corollary}\label{cor1}
Let $K$ be an infinite field of char $K\neq 2$. 
For a fixed $a\in K$ with $a^2+3a+9\neq 0$, there exist infinitely many 
$b\in K$ such that $\mathrm{Spl}_{K} f^{C_6}_b(X)=\mathrm{Spl}_K f^{C_6}_a(X)$. 
\end{corollary}
However, by applying Siegel's theorem for curves of genus $0$ (cf. \cite[Theorem 6.1]{Lan78}, 
\cite[Chapter 8, Section 5]{Lan83}) to Theorem \ref{thC6} (iii), we get
\begin{corollary}\label{cor3}
Let $K$ be a number field and $\mathcal{O}_K$ the ring of integers in $K$. 
Assume that $a\in \mathcal{O}_K$ with $a^2+3a+9\neq 0$. 
Then there exist only finitely many integers $b\in\mathcal{O}_K$ such that 
$\mathrm{Spl}_{K} f^{C_6}_b(X)=\mathrm{Spl}_K f^{C_6}_a(X)$. 
In particular, there exist only finitely many integers $b\in\mathcal{O}_K$ 
such that $f_{A_i}^{C_6}(X)$, $(i=1,2)$, has a linear factor over $\bQ$. 
\end{corollary}
When $K=\bQ$, by Okazaki's theorem (Theorem \ref{thOka}) and Theorem \ref{thC6} 
we have
\begin{theorem}\label{thisoC6}
Let $L^{(6)}_m=\mathrm{Spl}_\bQ f_m^{C_6}(X)$. 
For $m,n\in\bZ$, $L^{(6)}_m=L^{(6)}_n$ if and only if $m=n$ or $m=-n-3$. 
\end{theorem}
\begin{proof}
We should check the assertion only for $-1\leq m<n$ and 
$m,n\in\{-1$, $0$, $1$, $2$, $3$, $5$, $12$, $54$, $66$, $1259$, $2389\}$ 
because $L_m^{(3)}$ is the cubic subfield of $L_m^{(6)}$ and hence 
$L_m^{(6)}=L_n^{(6)}$ implies $L_m^{(3)}=L_n^{(3)}$ (cf. Theorem \ref{thOka}). 
The irreducible factorization of the corresponding multi-resolvent polynomials 
$\mathcal{R}^i_{f_{m,n}}(X)=f_{A_i}^{C_6}(X)$ over $\bQ$ are given as on Table $2$. 
\begin{center}
{\rm Table} $2$\vspace*{5mm}\\
{\renewcommand\arraystretch{1.2}
\begin{tabular}{|c|c|c|c|}\hline
$m$&$n$&$i$& irreducible factorization of $f_{A_i}^{C_6}(X)$ over $\bQ$\\\hline
$-1$&$5$&$1$&$(X^2-4X-3)(X^2+3X+\frac{1}{2})(X^2+\frac{2X}{3}-\frac{2}{3})$\\
$-1$&$12$&$2$&$(X^2-2X-2)(X^2+4X+1)(X^2+X-\frac{1}{2})$\\
$-1$&$1259$&$1$&$(X^2-\frac{5}{2}X-\frac{9}{4})(X^2+\frac{18}{5}X+\frac{4}{5})
(X^2+\frac{8}{9}X-\frac{5}{9})$\\
$5$&$12$&$2$&$(X^2-8X-5)(X^2+\frac{5}{2}X+\frac{1}{4})(X^2+\frac{2}{5}X-\frac{4}{5})$\\
$5$&$1259$&$1$&$(X^2-\frac{38}{3}X-\frac{22}{3})(X^2+\frac{44}{19}X+\frac{3}{19})
(X^2+\frac{3}{11}X-\frac{19}{22})$\\
$12$&$1259$&$2$&$(X^2-26X-14)(X^2+\frac{28}{13}X+\frac{1}{13})
(X^2+\frac{1}{7}X-\frac{13}{14})$\\\hline
$0$&$3$&$2$&$(X^2-2X-2)(X^2+4X+1)(X^2+X-\frac{1}{2})$\\
$0$&$54$&$1$&$(X^2-4X-3)(X^2+3X+\frac{1}{2})(X^2+\frac{2}{3}X-\frac{2}{3})$\\
$3$&$54$&$2$&$(X^2-8X-5)(X^2+\frac{5}{2}X+\frac{1}{4})
(X^2+\frac{2}{5}X-\frac{4}{5})$\\\hline
$1$&$66$&$2$&$(X^2-5X-\frac{7}{2})(X^2+\frac{14}{5}X+\frac{2}{5})
(X^2+\frac{4}{7}X-\frac{5}{7})$\\\hline
$2$&$2389$&$2$&$(X^2-7X-\frac{9}{2})(X^2+\frac{18}{7}X+\frac{2}{7})
(X^2+\frac{4}{9}X-\frac{7}{9})$\\\hline
\end{tabular}
}
\end{center}
Although we already know $L_0^{(6)}=L_0^{(3)}\neq L_5^{(3)}= L_5^{(6)}$, 
we do not omit the degenerate cubic case $m,n\in\{0,5\}$ on Table $2$. 
By Theorem \ref{thDecom}, the other sextic multi-resolvent polynomial 
$\mathcal{R}^j_{f_{m,n}}(X)=f_{A_j}^{C_6}(X)$, $(j\in\{1,2\}, j\neq i)$, 
is irreducible over $\bQ$. 
By Theorem \ref{thC6}, we conclude that the overlap $L_m^{(6)}=L_n^{(6)}$ 
occurs only for the trivial cases $m=n$ and $m=-n-3$. 
\end{proof}

%%%%%%%%%%%%%%%%%%%%%%%%%%%%%%%%%%%%%%%%%%%%%%%%%%%%%%%%%%%%%%%%%%%%%%%%%%%%%%%%%%%%%%%%%%%%
%
\section{Correspondence}\label{secores}
%
%%%%%%%%%%%%%%%%%%%%%%%%%%%%%%%%%%%%%%%%%%%%%%%%%%%%%%%%%%%%%%%%%%%%%%%%%%%%%%%%%%%%%%%%%%%%

The aim of this section is to establish the correspondence between 
isomorphism classes of the simplest sextic fields $L^{(6)}_m$ and 
non-trivial solutions to sextic Thue equations $F_m(x,y)=\lambda$ 
where $\lambda$ is a divisor of $27(m^2+3m+9)$ as follows 
(cf. cubic case \cite{H1} and quartic case \cite{H2}): 
\begin{theorem}\label{th2}
Let $m\in\bZ$ and $L^{(6)}_m=\mathrm{Spl}_\bQ f^{C_6}_m(X)$. 
There exists an integer $n\in\bZ\setminus\{m,-m-3\}$ such that 
$L^{(6)}_n=L^{(6)}_m$ if and only if there exists non-trivial solution 
$(x,y)\in\bZ^2$, i.e. $xy(x+y)(x-y)(x+2y)(2x+y)\neq 0$, to $F_m(x,y)=\lambda$ 
where $\lambda$ is a divisor of $27(m^2+3m+9)$. 
\end{theorem}
\begin{proof}
We apply Theorem \ref{thC6} to the case $K=\bQ$. 

Assume that $L^{(6)}_m=L^{(6)}_n$ for $n\in\bZ\setminus\{m,-m-3\}$. 
Then by Theorem \ref{thC6} (iii) with $z=x/y$, there exist $x,y\in\bZ$ with 
$\mathrm{gcd}(x,y)=1$ such that 
\begin{align}
N\,=\,m+\frac{(m^2+3m+9)xy(x+y)(x-y)(x+2y)(2x+y)}{F_m(x,y)}\in\bZ\label{eqNmZ}
\end{align}
where either $N=n$ or $N=-n-3$. 
The condition (\ref{eqNmZ}) occurs for only one of $N=n$ and $N=-n-3$ because 
only one of $f_{A_1}^{C_6}(X)$ and $f_{A_2}^{C_6}(X)$ in Theorem \ref{thC6} (ii) 
has a linear factor over $\bQ$. 
By (\ref{eqNmZ}), the assumption $n\in\bZ\setminus\{m,-m-3\}$ implies 
$xy(x+y)(x-y)(x+2y)(2x+y)\neq 0$.

Now we should show that $F_m(x,y)$ divides $27(m^2+3m+9)$. 

We use a standard method via resultant and the Sylvester matrix 
(cf. \cite{PV00}, \cite[Section 1.3]{SWP08}, see also \cite[Theorem 6.1]{Lan78}, 
\cite[Chapter 8, Section 5]{Lan83}). 
Put 
\[
h(z):=(m^2+3m+9)z(z+1)(z-1)(z+2)(2z+1)
\]
and take $f^{C_6}_m(z)=F_m(z,1)$. 
We take the resultant
\[
R_m:=\mathrm{Res}_z(h(z),f^{C_6}_m(z))=-3^9(m^2+3m+9)^6
\]
of $h(z)$ and $f^{C_6}_m(z)$ with respect to $z$. 
The resultant $R_m$ is also given by the determinant of the following modified 
Sylvester matrix of size $11\times 11$:  
\begin{align*}
S'(h,f_m^{C_6})=\left[
\begin{array}{lllllll}
 a_5 & a_4 & \cdots & a_0 & 0 & h(z)z^5 \\
 0 & \ddots & \ddots & \cdots & \ddots & \vdots \\
 0 & 0 & a_5 & a_4 & \cdots & h(z) \\
 b_6 & b_5 & \cdots & b_0 & 0 & f_m^{C_6}(z)z^4 \\
 0 & \ddots & \ddots & \cdots & \ddots & \vdots \\
 0 & 0 & b_6 & b_5 & \cdots & f_m^{C_6}(z) \\
\end{array}
\right]
\end{align*}
where $h(z)=\sum_{i=0}^5 a_iz^i$, $f^{C_6}_m(z)=\sum_{i=0}^6 b_iz^i$.
By the cofactor expansion along the 11th column of the matrix $S'(h,f^{C_6}_m)$, we have 
\begin{align}
h(z)(A_1z^5+\cdots+A_5z+A_6)+f^{C_6}_m(z)(A_7z^4+\cdots+A_{10}z+A_{11})=R_m\label{eqA11}. 
\end{align}
Dividing the both sides of (\ref{eqA11}) by 
$-\mathrm{gcd}(A_1,\ldots,A_{11})=-3^6(m^2+3m+9)^5$, we have 
\begin{align*}
h(z)p(z)+f^{C_6}_m(z)q(z)=27(m^2+3m+9)
\end{align*}
where
\begin{align*}
p(z)&=84z^5-42(4m+1)z^4-112(3m+11)z^3\\
&\quad +7(22m-153)z^2+2(161m+219)z+27m+242,\\
q(z)&=(m^2+3m+9)(-168z^4-336z^3+154z^2+322z+27).
\end{align*}
Put $H(x,y):=y^6 h(x/y)$, $P(x,y):=y^5 p(x/y)$, $Q(x,y):=y^5 q(x/y)$. 
Then it follows from $z=x/y$ and $F_m(x,y)=y^6 f^{C_6}_m(x/y)$ that 
\begin{align*}
H(x,y)P(x,y)+F_m(x,y)Q(x,y)=27(m^2+3m+9)y^{11}.
\end{align*}
Hence by (\ref{eqNmZ}) we have
\begin{align*}
\frac{H(x,y)P(x,y)}{F_m(x,y)}+Q(x,y)=\frac{27(m^2+3m+9)y^{11}}{F_m(x,y)}\in\bZ.
\end{align*}
Because the sextic forms $F_m(X,Y)$ and $H(X,Y)$ are invariants under the action of 
$\sigma\,:\, X\mapsto X+Y$, $Y\mapsto -X$, we may also get 
\begin{align*}
\frac{H(x,y)P(x+y,-x)}{F_m(x,y)}+Q(x+y,-x)
=\frac{27(m^2+3m+9)(-x)^{11}}{F_m(x,y)}\in\bZ.
\end{align*}
We conclude that $F_m(x,y)$ divides $27(m^2+3m+9)$ because $\mathrm{gcd}(x,y)=1$.

Conversely if there exists $(x,y)\in\bZ^2$ with $xy(x+y)(x-y)(x+2y)(2x+y)\neq 0$ 
such that $F_m(x,y)$ divides $27(m^2+3m+9)$ then we can take 
\begin{align*}
N\,=\,m+\frac{(m^2+3m+9)xy(x+y)(x-y)(x+2y)(2x+y)}{F_m(x,y)}\in\bQ\setminus\{m\}
\end{align*}
which satisfies $L^{(6)}_N=L^{(6)}_m$ by Theorem \ref{thC6}. 
It follows from $\mathrm{Gal}_\bQ f^{C_6}_m(X)\cong C_6$ or $C_3$ that 
$N\neq -m-3$ (see also Table $1$). 
Hence we have $N\in\bQ\setminus\{m,-m-3\}$. 

We see that $N\in\bZ$ as follows: 
If $x\equiv y\pmod{3}$ then $xy(x+y)(x-y)(x+2y)(2x+y)\equiv 0\pmod{27}$. 
Hence we have $N\in\bZ\setminus\{m,-m-3\}$. 

By a direct calculation, we obtain that if $x\not\equiv y\pmod{3}$ 
then $F_m(x,y)\equiv 1\pmod{3}$. 
Hence $F_m(x,y)$ divides $m^2+3m+9$ and $N\in\bZ\setminus\{m,-m-3\}$. 
\end{proof}

%%%%%%%%%%%%%%%%%%%%%%%%%%%%%%%%%%%%%%%%%%%%%%%%%%%%%%%%%%%%%%%%%%%%%%%%%%%%%%%%%%

%\begin{acknowledgment}
%This work is partially supported by Rikkyo University Special Fund for Research. 
%\end{acknowledgment}

%%%%%%%%%%%%%%%%%%%%%%%%%%%%%%%%%%%%%%%%%%%%%%%%%%%%%%%%%%%%%%%%%%%%%%%%%

\vspace*{10mm}
\noindent
Akinari Hoshi\\
Department of Mathematics\\
Rikkyo University\\ 
3--34--1 Nishi-Ikebukuro Toshima-ku\\ 
Tokyo, 171--8501, Japan\\
E-mail: \texttt{hoshi@rikkyo.ac.jp}\\
Web: \texttt{http://www2.rikkyo.ac.jp/web/hoshi/}


\begin{thebibliography}{Wak07b}

\bibitem[Ade01]{Ade01} C. Adelmann, {\it The decomposition of primes in torsion point fields}, 
Lecture Notes in Mathematics, 1761, Springer-Verlag, Berlin, 2001.

\bibitem[AHK98]{AHK98} H. Ahmad, M. Hajja, M. Kang, {\it Negligibility of projective 
linear automorphisms}, J. Algebra \textbf{199} (1998), 344--366.

\bibitem[Cha96]{Cha96} R. J. Chapman, {\it Automorphism polynomials in cyclic cubic 
extensions}, J. Number Theory \textbf{61} (1996), 283--291.

\bibitem[Coh93]{Coh93} H. Cohen, {\it A course in computational algebraic number theory}, 
Graduate Texts in Mathematics, 138, Springer-Verlag, Berlin, 1993. 

\bibitem[Coh00]{Coh00} H. Cohen, {\it Advanced topics in computational number theory}, 
Graduate Texts in Mathematics, 193, Springer-Verlag, New York, 2000. 

\bibitem[Enn91]{Enn91} V. Ennola, {\it Cubic number fields with exceptional units}, 
Computational number theory (Debrecen, 1989), 103--128, de Gruyter, Berlin, 1991. 

\bibitem[Gaa02]{Gaa02} I. Ga\'al, {\it Diophantine equations and power integral bases. 
New computational methods}, Birkh\"auser Boston, Inc., Boston, MA, 2002. 

\bibitem[GAP]{GAP} The GAP Group, GAP --- Groups, Algorithms, and Programming, Version 4.4.10; 
2007 (http://www.gap-system.org). 

\bibitem[Gir83]{Gir83} K. Girstmair, {\it On the computation of resolvents and Galois groups}, 
Manuscripta Math. \textbf{43} (1983), 289--307.  

\bibitem[Gra86]{Gra86} M. N. Gras, {\it Familles d'unit\'{e}s dans les extensions cycliques 
r\'{e}elles de degr\'{e} $6$ de $Q$}, (French) Th\'{e}orie des nombres, Ann\'{e}es 
1984/85--1985/86, Fasc. 2, Exp. No. 2, 27 pp., Publ. Math. Fac. Sci. Besan\c{c}on, Univ. 
Franche-Comt\'{e}, Besan\c{c}on, 1986. 

\bibitem[Gra87]{Gra87} M. N. Gras, {\it Special units in real cyclic sextic fields}, 
Math. Comp. \textbf{48} (1987), 179--182. 

\bibitem[HH05]{HH05} K. Hashimoto, A. Hoshi, {\it Families of cyclic polynomials 
obtained from geometric generalization of Gaussian period relations}, Math. Comp. 
\textbf{74} (2005), 1519--1530.

\bibitem[HM99]{HM99} K. Hashimoto, K. Miyake, {\it Inverse Galois problem for dihedral 
groups}, Number theory and its applications (Kyoto, 1997), 165--181,
Dev. Math., 2, Kluwer Acad. Publ., Dordrecht, 1999. 

\bibitem[H1]{H1} A. Hoshi, {\it On correspondence between solutions of a family of cubic Thue 
equations and isomorphism classes of the simplest cubic fields}, preprint, arXiv:0810.3374v3. 

\bibitem[H2]{H2} A. Hoshi, {\it On the simplest quartic fields and related Thue equations}, 
preprint, arXiv:1004.1960v2.

\bibitem[HM07]{HM07} A. Hoshi, K. Miyake, {\it Tschirnhausen transformation of a cubic 
generic polynomial and a $2$-dimensional involutive Cremona transformation}, 
Proc. Japan Acad. Ser. A \textbf{83} (2007), 21--26.

\bibitem[HM09a]{HM09a} A. Hoshi, K. Miyake, {\it A geometric framework for the subfield 
problem of generic polynomials via Tschirnhausen transformation}, 
Number theory and applications,  65--104, Hindustan Book Agency, New Delhi, 2009.

\bibitem[HM09b]{HM09b} A. Hoshi, K. Miyake, {\it On the field intersection problem of 
quartic generic polynomials via formal Tschirnhausen transformation}, 
Comment. Math. Univ. St. Pauli \textbf{58} (2009), 51--86. 

\bibitem[HM09c]{HM09c} A. Hoshi, K. Miyake, {\it On the field intersection problem of 
generic polynomials: a survey}, RIMS K\^oky\^uroku Bessatsu \textbf{B12} (2009), 231--247.

\bibitem[HM10a]{HM10a} A. Hoshi, K. Miyake, {\it Some Diophantine problems arising from 
the isomorphism problem of generic polynomials}, Number Theory: Dreaming in Dreams, 
87--105, Proceedings of the 5th China-Japan Seminar, World Sci. Publ., Singapore, 2010.

\bibitem[HM10b]{HM10b} A. Hoshi, K. Miyake, {\it A note on the field isomorphism problem 
of $X^3+sX+s$ and related cubic Thue equations}, Interdiscip. Inform. Sci. \textrm{16} 
(2010), 45--54.

\bibitem[HM10c]{HM10c} A. Hoshi, K. Miyake, {\it On the field intersection problem of 
solvable quintic generic polynomials}, Int. J. Number Theory \textbf{6} (2010), 1047--1081. 

\bibitem[Kid05]{Kid05} M. Kida, {\it Kummer theory for norm algebraic tori}, 
J. Algebra \textbf{293} (2005), 427--447.

\bibitem[Kom04]{Kom04} T. Komatsu, {\it Arithmetic of Rikuna's generic cyclic polynomial 
and generalization of Kummer theory}, Manuscripta Math. \textbf{114} (2004), 265--279. 

\bibitem[Lan78]{Lan78} S. Lang, {\it Elliptic curves: Diophantine analysis}, 
Grundlehren der Mathematischen Wissenschaften, 231, Springer-Verlag, Berlin-New York, 1978.

\bibitem[Lan83]{Lan83} S. Lang, {\it Fundamentals of Diophantine geometry}, 
Springer-Verlag, New York, 1983. 

\bibitem[LPV98]{LPV98} G. Lettl, A. Peth\"o, P. Voutier, {\it On the arithmetic of 
simplest sextic fields and related Thue equations}, Number theory (Eger, 1996), 
331--348, de Gruyter, Berlin, 1998. 

\bibitem[LPV99]{LPV99} G. Lettl, A. Peth\"o, P. Voutier, {\it Simple families of Thue 
inequalities}, Trans. Amer. Math. Soc. \textbf{351} (1999), 1871--1894. 

\bibitem[Lou07]{Lou07} S. R. Louboutin, {\it Efficient computation of root numbers and class 
numbers of parametrized families of real abelian number fields}, 
Math. Comp. \textbf{76} (2007), 455--473. 

\bibitem[Miy99]{Miy99} K. Miyake, {\it Linear fractional transformations and cyclic 
polynomials}, Algebraic number theory (Hapcheon/Saga, 1996), 
Adv. Stud. Contemp. Math. (Pusan) \textbf{1} (1999), 137--142.

\bibitem[Mor94]{Mor94} P. Morton, {\it Characterizing cyclic cubic extensions by automorphism 
polynomials}, J. Number Theory \textbf{49} (1994), 183--208.

\bibitem[Oga03]{Oga03} H. Ogawa, {\it Quadratic reduction of multiplicative group and 
its applications}, (Japanese) 
Algebraic number theory and related topics (Kyoto, 2002), 
S\=urikaisekikenky\=usho K\=oky\=uroku No. 1324, 217--224, 2003.

\bibitem[Oka02]{Oka02} R. Okazaki, {\it Geometry of a cubic Thue equation}, 
Publ. Math. Debrecen \textbf{61} (2002), 267--314. 

\bibitem[Oka]{Oka} R. Okazaki, {\it The simplest cubic fields are non-isomorphic 
to each other}, presentation sheet, 
available from \verb+http://www1.doshisha.ac.jp/~rokazaki/papers.html+. 

\bibitem[PV00]{PV00} D. Poulakis, E. Voskos, {\it On the practical solution of genus 
zero Diophantine equations}, J. Symbolic Comput. \textbf{30} (2000), 573--582.

\bibitem[RV99]{RV99} N. Rennert and A. Valibouze, {\it Calcul de r\'esolvantes avec 
les modules de Cauchy}, Experiment. Math. \textbf{8} (1999), 351--366.

\bibitem[Ren04]{Ren04} N. Rennert, {\it A parallel multi-modular algorithm for computing 
Lagrange resolvens}, J. Symbolic Comput. \textbf{37} (2004), 547--556. 

\bibitem[Rik02]{Rik02} Y. Rikuna, {\it On simple families of cyclic polynomials}, 
Proc. Amer. Math. Soc. \textbf{130} (2002), 2215--2218. 

\bibitem[SWP08]{SWP08} J. R. Sendra, F. Winkler, S. P\'erez-D\'iaz, 
{\it Rational algebraic curves. A computer algebra approach}, 
Algorithms and Computation in Mathematics, 22. Springer, Berlin, 2008. 

\bibitem[Sha74]{Sha74} D. Shanks, {\it The simplest cubic fields}, 
Math. Comp. \textbf{28} (1974), 1137--1152.

\bibitem[Tog02]{Tog02} A. Togb\'e, {\it On the solutions of a family of sextic Thue 
equations},  Number theory for the millennium, III (Urbana, IL, 2000), 
285--299, A K Peters, Natick, MA, 2002. 

\bibitem[Wak07a]{Wak07a} I. Wakabayashi, {\it Number of solutions for cubic Thue equations 
with automorphisms}, Ramanujan J. \textbf{14} (2007), 131--154. 

\bibitem[Wak07b]{Wak07b} I. Wakabayashi, {\it Simple families of Thue inequalities}, 
Ann. Sci. Math. Qu\'ebec \textbf{31} (2007), 211--232.

\end{thebibliography}
\end{document}